\documentclass[10pt,reqno]{amsart}
\usepackage[centertags]{amsmath}

\theoremstyle{plain} 
\newtheorem{cor}{Corollary}

\newtheorem{thm}{Theorem}

\theoremstyle{definition}

\theoremstyle{remark}

\numberwithin{equation}{section}

\DeclareMathSymbol{\R}{\mathalpha}{AMSb}{"52}
\DeclareMathSymbol{\C}{\mathalpha}{AMSb}{"43}

\newcommand{\beq}{\begin{equation}}
\newcommand{\eeq}{\end{equation}}

\newcommand{\set}[1]{\left\{#1\right\}}

\newcommand{\pd}{\,\partial}

\newcommand{\bd}{\begin{description}}
\newcommand{\ed}{\end{description}}
\newcommand{\beqr}{\begin{eqnarray}}
\newcommand{\eeqr}{\end{eqnarray}}
\newcommand{\beqt}{\begin{equation}}
\newcommand{\eeqt}{\end{equation}}

\begin{document}

\title[On generating  relative and absolute invariants]
{On generating  relative and absolute invariants of linear
differential equations}
\author[]{J C Ndogmo}

\address{PO Box 2446
Bellville 7535\\
South Africa\\
}
\email{ndogmoj@yahoo.com}

\begin{abstract}
A general expression for a relative invariant of a linear ordinary
differential equations is given in terms of the fundamental
semi-invariant and an absolute invariant. This result is used to
established a number of properties of relative invariants, and it is
explicitly shown how to generate fundamental sets of relative and
absolute invariants of all orders for the general linear equation.
Explicit constructions are made for the linear ODE of order five.
The approach used for the explicit determination of invariants is
based on an infinitesimal method.
\end{abstract}


%
%
\maketitle

\section{Introduction}
\label{s:intro}The first invariants found during the early days of
development of the theory of invariants of differential equations
were all relative invariants ~\cite{lapla, lag, brio}, and one of
the very first determination of absolute invariants for differential
equations is perhaps due to Brioschi ~\cite{brio}, who obtained them
as a quotient of two relative invariants. It has subsequently been
shown that in the case of linear ordinary differential equations,
every fundamental absolute invariant can be expressed as a rational
function.\par

These basic properties of relative invariants have led to
significant progress in the study of invariants of differential
equations and their applications, and such studies have been largely
influenced by the work of Halphen ~\cite{halph82, halph66}, and
Forsyth ~\cite{for-inv}. However the methods used by Halphen,
Forsyth, and earlier researchers on the subject  were very intuitive
and most often ad hoc methods requesting tedious calculations just
for finding a couple of invariants.\par
  Based on recent advances in Lie group techniques
 ~\cite{ovsy1, ibra-nl, olvgen, ndogftc}, infinitesimal methods haven
been increasingly used for the investigation of invariants of
differential equations ~\cite{ibra-par, faz,waf,ndogschw, schw,
melesh}. But although these infinitesimal methods provide a more
systematic route for the treatment of invariants of differential
equations and transformation groups, they have been applied to the
determination of relative invariants of differential equations only
in some very rare cases ~\cite{ibra-lap, faz}, and even in those
cases only some very specific relative invariants were obtained in
the usual way as absolute invariants corresponding to partial
structure-preserving transformations, in which some of the variables
in the equation are kept constant. It therefore appears that a
number of interesting properties of these relative invariants have
not yet been uncovered.
\par

Restricting our attention to linear ordinary differential equations,
we determine in this paper some properties of relative invariants,
considered as semi-invariants of the full structure-preserving
transformations of these equations. In particular we show that every
absolute invariant can be expressed as a quotient of a relative
invariant and the fundamental semi-invariant, and we derive a
general expression for these relative invariants. We show how these
relative invariants can be used to obtain invariants of all orders
via invariant differentiation. Our approach for any explicit
determination of invariants is based on the infinitesimal method
recently proposed in
 ~\cite{ndogftc}, and which contrary to the former well-known method
of ~\cite{ibra-nl}, does not require the knowledge of the
structure-preserving transformations, but rather provides it.

\section{Basic properties of relative invariants}\label{s:basic}
 Let $G$ be a Lie group of point transformations of the form
\begin{equation}\label{eq:invgp}
x= \phi (z, w; \tau), \qquad y = \psi ( z, w; \tau),
\end{equation}
 where $\tau$ denotes collectively some arbitrary
parameters specifying the group element in $G$. Consider on the
other hand  a family $ \mathcal{D}$ of differential equations of the
general form
\begin{equation}\label{eq:gnl}
\Omega(x, y_{(n)}; \rho)=0,
\end{equation}
in which $y_{(n)}$ represents the dependent variable $y=y(x)$ and
all its derivatives up to the order $n$, and  $\rho$ represents
collectively some arbitrary functions of $x,$ or arbitrary constants
specifying the family element in $\mathcal{D}.$  We say that $G$ is
the \emph{equivalence group} of ~\eqref{eq:gnl} if it is the largest
group of transformations that maps elements of $\mathcal{D}$ into
itself. In this case, the transformed equation takes the same form
$\Omega(z, w_{(n)}; \theta)=0,$ in terms of the transformed
parameter $\theta,$ and ~\eqref{eq:invgp} is called the
\emph{structure-preserving transformations} of ~\eqref{eq:gnl}. By a
well known result of Lie ~\cite{liegc}, the transformations
 ~\eqref{eq:invgp} induces another group of transformation $G_c$
acting on the parameter  $\rho$ of the differential equation, and we
shall be interested in the invariants of this group action and its
prolongations, and which are commonly referred to as the invariants
(or differential invariants) of ~\eqref{eq:gnl}. As sets, $G$ and
$G_c$ are equal, except that they act on different spaces, and both
are often referred to as the equivalence group of ~\eqref{eq:gnl}.
Thus both sets will often be denoted simply by $G.$
\par

 It should be noted that in ~\eqref{eq:gnl}, $x$
and $y$ each denote collectively all independent and dependent
variables, respectively, so that the equation also includes in
particular all partial differential equations. However, we shall be
interested in this paper in a linear ordinary differential equations
of the general form
\begin{equation}\label{eq:gnlin}
y^{(n)}+ a_1 y^{(n-1)} + a_2 y^{(n-2)}+ \dots + a_n y=0,
\end{equation}
where $a_j= a_j(x)$ are arbitrary functions of the independent
variable $x.$ The structure-preserving transformations of
~\eqref{eq:gnlin} can be written in the form
\begin{equation}\label{eq:gnlingp}
x= \xi (z), \qquad y= \eta(z) w,
\end{equation}
where $\xi$ and $\eta$ are arbitrary functions. Under
 ~\eqref{eq:gnlingp}, the transformed equation of ~\eqref{eq:gnlin}
takes the form
\begin{equation}\label{eq:gnlin2}
w^{(n)} + A_1 w^{(n-1)}+ A_2 w^{(n-2)} + \dots+ A_n w=0,
\end{equation}
 where the $A_j= A_j(z)$ are the new coefficients. Let $a$ denote
collectively the coefficients $a_j= a_j(x).$  A differential
function $F= F(a, a_{(r)}),$ where $a_{(r)}$ represents as usual the
derivative of $a$ up to a certain order $r,$ is called a
\emph{relative invariant} of ~\eqref{eq:gnlin} if
\begin{align}\label{eq:dfnsemi}
 F(\tau \cdot (a, a_{(r)}))& =
\mathbf{w}(\tau)\cdot F(a, a_{(r)}),
\end{align}
for all $a$ and $r,$ and for all $\tau \in G.$ In
 ~\eqref{eq:dfnsemi}, the weight function $\mathbf{w}=
\mathbf{w}(\tau)$ must  be a character of the group $G.$ When
$\mathbf{w}$ is identically equal to one, the function $F$ is called
an \emph{absolute invariant} of ~\eqref{eq:gnlin}. For simplicity,
an expression of the form $F(\tau \cdot (a, a_{(r)}))$ like in
 ~\eqref{eq:dfnsemi} will often be represented by $F_\tau,$ for a
given function ~$F.$

\par
Let's assign to an expression of the form $d^k a_j/d x^k$ the
\emph{weight} $j+k$. We say that a polynomial function $F$ in the
coefficients $a_j$ and their derivatives has weight $m$ if each of
its terms has constant weight $m.$ By combining  a result of Forsyth
 ~\cite{for-inv} according to which all absolute
invariants of ~\eqref{eq:gnlin} can be obtained as rational
functions, and certain results obtained by Halphen ~\cite{halph66},
we readily obtain the following result.
\begin{thm}\label{th:quotient}
Equation ~\eqref{eq:gnlin} has a fundamental set of absolute
invariants consisting of rational functions, in which every element
$F= S_1/S_2$ is the quotient of two relative invariants $S_1$ and
$S_2$ of the same weight $m,$ each of which satisfies a relation of
the form
\begin{equation}\label{eq:quotient}
S(\tau \cdot (a, a_{(r)})) =\xi(z)^{m} S(a, a_{(r)}),\quad \text{
that is } \quad S_\tau= \xi(z)^m S,
\end{equation}
for every $\tau \in G,$ where $S$ denotes any of the invariants
$S_1$ and $S_2.$
\end{thm}
The  invariant $F$ of Theorem ~\ref{th:quotient} is also said to be
of weight $m.$ In general, a relative invariant $S$ satisfying
$S_\tau= \xi^r S$ for some $r \in \R$ is said to be of \emph{index}
$r.$ We shall make the result of this theorem more precise in the
next section.
\section{The fundamental relative invariant}
\label{s:fdamental} For a general group of transformations $G$
acting on a manifold $M$, a semi-invariant usually refers to a
function $F$ satisfying a relation similar to that specified by
 ~\eqref{eq:dfnsemi}, and which has the simple form
\begin{equation}\label{eq:dfnsemi2}
F(g \cdot p)= \mathbf{w}(g) \cdot F(p),
\end{equation}
for all $g\in G$ and $p \in M$ such that $g \cdot p$ is defined, and
where the function $\mathbf{w}$ is a character of $G.$ If we let $v$
be a generic element in the generating system for the Lie algebra of
$G$ and denote by $X$ the corresponding infinitesimal generator of
the group action, then ~\eqref{eq:dfnsemi2} implies that
\begin{subequations}\label{eq:infisemi}
\begin{align}
X \cdot F &= - \lambda F, \qquad \text{where} \\
  \lambda &= d\, \mathbf{w}(e) (v),
\end{align}
\end{subequations}
and where $d\, \mathbf{w}(e)$ denotes the differential of
$\mathbf{w}$ at the identity element $e$ of $G.$ A function $F$
satisfying ~\eqref{eq:infisemi} is often called a
\emph{$\lambda$-semi-invariant}. By a nontrivial semi-invariant, we
shall mean a semi-invariant which is not an absolute one, and which
in particular is not a constant function. We now establish the
following result that relates every semi-invariant to a fundamental
set of absolute invariants.
\begin{thm}\label{th:F0}
A function $F$ defined on $M$ is a $\lambda$-semi-invariant of $G$
if and only if it is of the form
\begin{equation}\label{eq:F0}
F= F_0 \Phi,
\end{equation}
where $F_0$ is an arbitrarily chosen  nontrivial
$\lambda$-semi-invariant, and $\Phi$ is an absolute invariant, that
is $\Phi= \Phi(I_1, \dots, I_s),$ where $\set{I_1, \dots, I_s}$ is a
fundamental set of absolute invariants of $G.$
\end{thm}
\begin{proof}
Suppose that in a coordinates system $\set{u_1, \dots, u_q}=u$ of
$M$ the infinitesimal generator $X$ of $G$  defined as in
 ~\eqref{eq:infisemi} has the form
$$ X= X_1 \pd_{u_1} + \dots + X_q \pd_{u_q},$$
where $X_j= X_j(u).$ Then by a well-known result ~\cite{sned}, the
equivalent system of characteristic equations associated with
 ~\eqref{eq:infisemi} is given by the sequence of $q$ equalities
\begin{equation}\label{eq:charsm}
\frac{d u_1}{X_1}= \frac{d u_2}{X_2} = \dots = \frac{d
u_q}{X_q}=\frac{d F}{ -\lambda F},
\end{equation}
in which the first $q-1$ equalities are the determining equations
for the absolute invariants. To find the general solution to this
system, we can first find the integral resulting from a combination
of the last fraction $d F/ (- \lambda F)$ with any suitable ones
from among the first $q$ fractions. We may assume without loss of
generality that such a combination is given by an equation of the
form
$$  \frac{d F}{ F} = \frac{-\lambda d u_k}{X_k},\qquad
\text{for some $k,\; 1\leq k \leq q$}.$$
This equation clearly has separable variables, since $F$ is
considered as a new independent variable added to the coordinates
system $\set{u_1, \dots, u_q}.$  Its integral can be written in the
form $F \nu= C_{q+1},$ for some function $\nu= \nu(u).$ If we denote
by $I_j(u)= C_j,$ for $j=1, \dots, q-1,$ the other integrals
corresponding to the first $q-1$ equalities in ~\eqref{eq:charsm},
and in which the $C_j$ are arbitrary constants, then the $I_j$ are
the absolute invariants of $G$ and  the general solution of
 ~\eqref{eq:charsm} takes the form
$$ \Phi_1(I_1, \dots, I_q, F \nu)=0, \quad \text{or equivalently } F \nu=
\Phi (I_1, \dots, I_q), $$
for some arbitrary functions $\Phi_1$ and $\Phi.$ To obtain the
function $F_0$ of the theorem we only need to take $F_0= 1/\nu,$ and
it readily follows from the Leibnitz property of the derivation $X$
and the definition of an absolute invariant that $F_0$ is a
$\lambda$-semi-invariant.
\end{proof}
It should be noted that Eq. ~\eqref{eq:F0} also holds with the same
$F_0$ for all prolongations of the group $G.$ Since $F_0$ is
nontrivial, it follows in particular that the set $\set{I_1, \dots,
I_q, F_0}$ is functionally independent. In contrast to the case of
absolute invariants, not every function of semi-invariants is again
a semi-invariant, and although the set of all semi-invariants of $G$
forms a group under functions multiplication, the sum or difference
of two semi-invariants is not  in general a semi-invariant.\par
 For simplicity, we let $I$ denote
collectively all elements in a fundamental set of absolute
invariants of Eq. ~\eqref{eq:gnlin}. We shall also often denote by
$\omega$ an element of the form $(a, a_{(s)}),$ which can be viewed
as an element in the $s$th-jet space determined by the independent
variable $x$ and the dependent variable $a$ of Eq.
~\eqref{eq:gnlin}.
\begin{cor}\label{co:1}
Suppose that $S_1$ is a relative invariant of Eq. ~\eqref{eq:gnlin}
of index $k.$ Then every relative invariant $S_2$ of order $m$ of
the same equation can be put into the form
\begin{equation}\label{eq:gnlsemi}
S_2= S_1^{m/k} \Phi,
\end{equation}
for a certain  function $\Phi= \Phi(I).$ In particular $S_1^m/S_2^k$
is an absolute invariant.
\end{cor}
\begin{proof}
According to Theorem ~\ref{th:F0}, to prove the first part of the
corollary,  we only need to show that $S_1^{m/k}$ is a
$\lambda$-semi-invariant for the same function $\lambda$ as $S_2,$
and by ~\eqref{eq:infisemi}, we only need to show that $S_1^{m/k}$
is a relative invariant of the same index $m$ as $S_2.$ But since
$(S_1)_\tau = \xi^k S_1$ by assumption, it readily follows that
$(S_1^{m/k})_\tau = \xi^m S_1^{m/k}.$ For the second part we notice
that since $S^p$ has index $r\!\,p$ for every relative invariant $S$
of index $r,$ we must have $(S_1^m/S_2^k)_\tau = S_1^m/ S_2^k.$
\end{proof}
The corollary says that every relative invariant of
~\eqref{eq:gnlin} can be expressed as a product of an arbitrarily
chosen, but fixed relative invariant and an arbitrary function of
$I.$ We call such a fixed relative invariant the \emph{fundamental
relative invariant} of ~\eqref{eq:gnlin}, and denote it and its
index by $S_0$ and $\sigma,$ respectively. We can now give without
any need of calculations a simple formula relating every absolute
invariant with the fundamental relative invariant.
\begin{cor}\label{co:2}
Eq. ~\eqref{eq:gnlin} has a fundamental set of absolute invariants
in which every element has the form $F = S_1 / S_0^{m /\sigma}$ for
some relative invariant $S_1$ of index $m.$
\end{cor}
\begin{proof}
By Theorem ~\ref{th:quotient}, we may write every fundamental
absolute invariant of index $m$ in the form $R_1/ R_2,$ where $R_1$
and $R_2$ are relative invariant of index $m.$ It follows from
Corollary ~\ref{co:1} that $R_1= S_0^{m/ \sigma} \Phi_1$ and $R_2=
S_0^{m/ \sigma} \Phi_2,$  for some functions $\Phi_1$ and $\Phi_2.$
The  corollary is thus proved by  taking $S_1= S_0^{m/ \sigma}
\Phi_1 / \Phi_2.$
\end{proof}
%
%
\begin{cor}\label{co:3}
The weight and the index of every relative invariant coincide.
\end{cor}
\begin{proof}
Suppose that the relative invariant $S$ has weight $\mu$ and index
$m$. Then $S/ S_0^{m / \sigma}$ is an absolute invariant whose
weight is the index  of $S_0^{m/\sigma},$  which is $m$ by Theorem
 ~\ref{th:quotient}. By the same theorem, this weight is the same as
the weight $\mu$ of $S,$ and this proves the corollary.
\end{proof}

\begin{cor}\label{co:4}\parbox[]{1in}{$\quad$}
\begin{enumerate}
\item[(a)] Suppose that $\set{I_1, \dots, I_p}$ is a fundamental set
of absolute invariants of a certain group of equivalence
transformations $G,$ and that $F_0$ is a nontrivial
$\lambda$-semi-invariant of $G.$ Then, $\set{F_0, F_0 I_1, \dots,
F_0 I_p}$ is a maximal set of functionally independent
$\lambda$-semi-invariants of $G.$

\item[(b)] If $I_j= S_j/ S_0^{m_j/\sigma}$ are the fundamental
absolute invariants of Eq. ~\eqref{eq:gnlin} for $j=1, \dots, p,$
where $m_j$ is the index of $S_j,$ then
$$\mathcal{S}=
\set{S_0^{m/\sigma}, S_1^{m/m_1}, \dots, S_p^{m/m_p}}
$$
is a fundamental set of relative invariants of index $m$ of Eq.
 ~\eqref{eq:gnlin}.
\end{enumerate}
\end{cor}

\begin{proof}
Since $F_0$ is a nontrivial semi-invariant of $G,$ the given set
$\set{F_0, F_0 I_1,  \dots, F_0 I_p}$ clearly forms a functionally
independent set of semi-invariants of $G,$ by Theorem ~\ref{th:F0}.
By the same theorem, any other semi-invariant of $G$ has the form
$S= F_0 \Phi,$ where $\Phi= \Phi (I_1, \dots, I_p),$ and this
readily proves part (a) of the corollary.\par
On the other hand, every element in the given set $\mathcal{S}$ is
clearly a relative invariant of index $m.$ if we replace each $I_j$
by
$$I_j^{ m / m_j} = S_j^{m/m_j} / S_0^{m/\sigma},$$
the resulting set formed by the $I_{j}^{m/m_j}$ is  a fundamental
set of absolute invariants of the same index $m$ of Eq.
 ~\eqref{eq:gnlin}. Thus by ~\eqref{eq:infisemi} the elements of
$\mathcal{S}$ are $\lambda$-semi-invariants corresponding to the
same function $\lambda.$ Therefore, part (b) of the corollary
readily follows from part (a).
\end{proof}
It also follows from part (b) of this corollary that
$\set{S_0,S_1,\dots,  S_p}$ is a functionally independent set of
relative invariants, but whose elements do not have the same index.

\section{Application to the determination of invariants}
\label{s:appl} It is well-known, by a result of Lie, that all higher
order differential invariants of a transformation group can be
obtained from a generating system of lower order ones by means of
invariant differentiation. This often reduces the problem of
determination of invariants  to finding a generating system  of
invariants and the invariant differential operators. To begin with,
suppose that we know two functionally independent absolute
invariants $I_0$ and $I_1$  of equation ~\eqref{eq:gnlin}, and that
these are rational functions of the form
\begin{equation}\label{eq:I0}
I_0= R_0^\sigma/S_0^k, \qquad I_1= S_1^\sigma/S_0^{m}
\end{equation}
where $S_0$ is the fundamental relative invariant of index $\sigma$
already introduced, while $R_0$ and $S_1$ are relative invariants of
indices $k$ and $m,$ respectively. We also assume that $S_1$ is of
order $\mu$ as a function of the independent variable $x,$ and that
all other relative invariants are of order at most $\mu$ in $x.$
Then by means of the differential operator $\zeta D_x,$ where
$\zeta= 1/I_0', \; I_0' = d I_0/dx$ and  $D_x= d/dx,$ we can
generate a new absolute invariant $I_1^*= \zeta D_x(I_1),$ whose
order is $\mu+1$ in general. In terms of the relative invariants in
 ~\eqref{eq:I0}, we have
\begin{equation}\label{eq:I1p}
I_1^* =  \frac{I_1}{I_0}\, \frac{R_0}{S_1}\, \frac{(m S_1 S_0' - r
S_0 S_1')}{(k R_0 S_0' - r S_0 R_0{\,\!\! '})},
\end{equation}
where $h' \equiv d h /d x$ for every function
 $h = h(x).$ It follows from Corollary ~\ref{co:2} that to find an absolute
invariant of higher order $\mu+1,$ we only need to find a relative
invariant of order $\mu+1.$ We now introduce the notation

\begin{subequations}\label{eq:varphi}
\begin{alignat}{2}
\varphi(R_1, R_2) &= m_1 R_1 R_2' - m_2 R_2 R_1', \quad  &\chi(R_1,
R_2)&= \frac{\left[\varphi (R_1, R_2)\right]^{m_2}}{R_2^{m_1+
m2+1}}\\
\varphi_0(R_1, R_2) &= R_1, \quad  &\chi_0(R_1, R_2) &=
R_1^{m_2}/R_2^{m_1}
\end{alignat}
\end{subequations}
for any relative invariants $R_1$ and $R_2$ of respective index
$m_1$ and $m_2.$ For simplicity of notation, when the second
argument $R_2$ is fixed and there is no possibility of confusion, we
shall set $\varphi(R_1)= \varphi(R_1, R_2)$ and $\chi
(R_1)=\chi(R_1, R_2).$ By multiplying $I_1^*$ by the absolute
invariant $I_0/ I_1$ and the relative invariant $S_1/ R_0,$ we
obtain a relative invariant of the form $F= \varphi(S_1,
S_0)/\varphi(R_0, S_0).$ By an application of Theorem ~\ref{th:F0},
it is easy to see that both $\varphi(S_1)$ and $\varphi(R_0)$ are
relative invariants, and they clearly have indices $m+r+1$ and
$k+r+1,$ respectively. Moreover, $\varphi (S_1)$ has the required
order $\mu+1,$ and it gives rise to the absolute invariant
\begin{equation}\label{eq:chi}
\chi (S_1,
S_0)=\frac{\left[\varphi(S_1)\right]^\sigma}{S_0^{(m+r+1)}}.
\end{equation}
Owing to the arbitrariness of the function $S_1,$ Eqs.
 ~\eqref{eq:varphi} and ~\eqref{eq:chi} show that starting with any
relative invariant $S_1$ of order $\mu,$ we can construct an
indefinite  sequence $\varphi_q (S_1)= \varphi^q (S_1)$ of relative
invariants and $\chi_q(S_1)= \chi^q (S_1)$ of absolute invariants,
each having a general term of order $\mu+q$. To write down the
expression for the general terms of these sequences, we only need to
note that $\varphi_q (S_1)$ has index $\theta (q)= m+ q (\sigma+1).$
Consequently, we have
\begin{align}\label{eq:varphiseq}
  \varphi_q (S_1)&= \theta (q-1) \varphi^{q-1}(S_1) S_0' - r S_0\, \varphi^{q-1}(S_1)', \qquad \chi_q
  (S_1)= \frac{\left[\varphi_q(S_1)\right]^\sigma}{  S_0^{\theta (q)}},\\
\intertext{ where } \varphi^{q-1}(S_1)'&= d (  \varphi^{q-1}(S_1))/
dx. \notag
\end{align}
A similar sequence of relative and absolute invariants can be
obtained using the relative invariant $\varphi(R_0, S_0)$. \par
In the case of ODEs only one differential operator of the form
$\zeta D_x$ is required and once this operator is known, finding
fundamental sets of invariants  only requires a suitable generating
system of invariants, and these invariants must be explicitly
calculated. As already mentioned, earlier methods for finding these
invariants such as those used in ~\cite{halph66} were very intuitive
and quite informal, and they were designed only for linear
equations. However, based on ideas suggested by Lie ~\cite{lie1},
and starting with some works undertaken by Ovsyannikov
~\cite{ovsy1}, the development of infinitesimal methods started to
grow and a number of Lie groups techniques for invariant functions
have been proposed in the recent scientific literature
 ~\cite{ibra-nl, olvmf, olvgen, ndogftc}. These Lie groups methods
provide a systematic means for finding invariant functions and some
of them have far reaching immediate applications (see ~\cite{olvmf,
ndogftc}). For instance, the method of ~\cite{ndogftc} that we shall
use for the explicit determination of the invariants also provides,
in infinitesimal form,  the structure-preserving transformations of
any differential equation. When it is applied to Eq.
~\eqref{eq:gnlin}, the coefficients $a_j(x)$ of the equation are
considered as dependent variables on the same footing as $y$ and the
infinitesimal generators $X^0$ of the equivalence group $G_c$ takes
the form
\begin{equation}\label{eq:infign}
X^0 = f \pd_{x} + \sum_{i=1}^n \phi_i \pd_{a_i}.
\end{equation}
The explicit expression  of $X^0$ depends on various canonical forms
adopted for the general linear ODE. By a change of variable of the
form $x=z$ and $y= \exp\left(- \int a_1 dz \right) w,$ Eq.
 ~\eqref{eq:gnlin} can be put in the form
\begin{equation}\label{eq:nor}
w^{(n)} + b_2 w^{(n-2)} + \dots + b_n w=0
\end{equation}
which is deprived of the term of second highest order, and in which
the $b_j = b_j(z)$ are the new coefficients. Similarly, by a change
of variables of the form
\begin{subequations}\label{eq:chg2schw}
\begin{align}
\set{z, x}&= \frac{12}{n (n-1)(n+1)}  a_{2}, \qquad y =
\exp\left(- \int a_1 d z\right) w,\\
\intertext{ where } \set{z, x} &= \left(z\,' z\,^{(3)} - (3/2)
z\,''^{\,2} \right)\, z\,'^{\,-2}
\end{align}
\end{subequations}
is the Schwarzian derivative, and where $z\,' = d z/ dx,$ we can put
Eq. ~\eqref{eq:gnlin} into a form in which the terms of orders $n-1$
and $n-2\,$ do not appear. After the renaming of variables and
coefficients of the new equation using the same notation as for the
original equation ~\eqref{eq:gnlin}, the transformed equation takes
the form
\begin{equation}\label{eq:schw}
y^{(n)} + a_3 y^{(n-3)} + \dots + a_n y=0.
\end{equation}
In fact, due to the prominence in size of invariants of differential
equations and the rapid rate at which this size increases with the
number of nonzero coefficients in the equation, it is customary to
use the canonical form ~\eqref{eq:schw} of ~\eqref{eq:gnlin} for the
investigation of invariants ~\cite{for-inv, brio}.\par
A determination of invariants of linear ODEs by means of the
infinitesimal generator of the form ~\eqref{eq:infign} was made in
 ~\cite{ndogschw} for equations of order up to five in the various
canonical forms ~\eqref{eq:gnlin}, ~\eqref{eq:nor}, and
 ~\eqref{eq:schw}, but only for low orders of prolongation of the
operator $X^0$ not exceeding three. We now undertake the explicit
determination of a generating system of invariants of all orders by
an application of formulas ~\eqref{eq:varphi}-\eqref{eq:varphiseq}.
For this purpose we shall adopt the canonical form ~\eqref{eq:schw}
for which the invariants have a much simpler expression, and we
restrict our attention to the case of equations of order $n=5$ which
have the  form
\begin{equation}\label{eq:sh5}
y^{(5)} + a_3 y''+ a_4 y' + a_5 y=0.
\end{equation}
For every $n>3$, and for any order $p$ of prolongation of the
operator $X^0,$ the number $\Gamma$ of absolute invariants of Eq.
\eqref{eq:schw} can be shown ~\cite{ndogschw} to be given by
\begin{equation}\label{eq:nbschw}
    \Gamma = n+4 - p(n-2),
\end{equation}
and this indicates that for a fixed value of $n,$ as the order of
prolongation of $X^0$ increases by one unit, the number of absolute
invariants increases by $n-2,$ which is precisely the number of
coefficients in the equation. This means that a generating system of
absolute invariants should contain $n-2$ elements, which by virtue
of ~\eqref{eq:nbschw} can be taken to be functionally independent,
provided that none of the coefficients $a_j$ vanishes. For
 ~\eqref{eq:sh5},  $X^0$ depends on three arbitrary constants $k_1,
k_2$ and $k_3$ and has the explicit form
\begin{equation}\label{eq:ifsh5}
\begin{split} X^0 &= \left[ k_1+ x (k_2 + k_3 x)\right] \pd_x - 3 a_3 (k_2 + 2 k_3 x)\pd_{a_3} \\
 & \quad -2 \left[-3 a_3 k_3 +2 a_4 (k_2 + 2 k_3 x)\right]\pd_{a_4} + \left[ -4 a_4 k_3
 - 5 a_5 (k_2+ 2 k_3 x)  \right] \pd_{a_5}\end{split}
\end{equation}
and its first prolongation has four invariants given by
\begin{subequations}\label{eq:ivp1sh5}
\begin{align}
I_0 &=  \frac{\left(3a_5\, a_3 - a_4^2  \right)^3}{27 a_3^8} \label{eq:schp0}\\
I_1 &= \frac{(- a_4 + a_3')^3}{a_3^4} \\
I_2 &=  \frac{\left( 6 a_3 a_4' - a_4^2 - 6 a_4\, a_3'\right)^3}{216\, a_3^8}\\
I_3 &=\frac{5 a_4^3 + 9 a_3^2\, a_5' - 3 a_4\, a_3(5 a_5  + 2 a_4')+
3 a_4^2 \, a_3'}{9 a_3^4}
\end{align}
\end{subequations}
The fundamental relative invariant of ~\eqref{eq:schw} is readily
seen to be given by $S_0=a_3(x),$ and Eq. ~\eqref{eq:ivp1sh5} shows
that the four functions
\begin{align*}
R_0 &= 3a_5\, a_3 - a_4^2  \\
S_1 &=- a_4 + a_3'  \\
S_2 &= 6 a_3 a_4' - a_4^2 - 6 a_4\, a_3' \\
S_3 &= 5 a_4^3 + 9 a_3^2\, a_5' - 3 a_4\, a_3(5 a_5  + 2 a_4')+ 3
a_4^2 \, a_3'
\end{align*}
are relative invariants of respective index $8, 4, 8,$ and $12.$  In
terms of these relative invariants, we have
\begin{align}\label{eq:ifsh5sm}
I_0 &=  \frac{R_0^3 }{ S_0^8},\qquad  I_1 = \frac{S_1^3 }{ S_0^4},
\qquad I_2 = \frac{S_2^3 }{ S_0^8},\qquad  I_3 = \frac{S_3}{ S_0^4}.
\end{align}
Moreover, each of the four functions $R_0, S_1, S_2,$ and $S_3$ is
of order at most one, and together they form a functionally
independent set. Consequently, on in accordance with
 ~\eqref{eq:varphiseq}, $\mathcal{B}= \set{I_0, I_1, I_2, I_3}$ forms
a functionally independent generating system of absolute invariants
which may be called a basis of absolute invariants of
 ~\eqref{eq:schw}.\par

  It also follows that a fundamental set of absolute
invariants of order $p \geq 1$ is given by the functions
\begin{equation}\label{eq:ifsh5bs}
I_0, \quad \chi_{k} (S_1),\quad \chi_{k} (S_2),\quad \chi_k
(S_3),\quad  \text{ for $k=0, \dots, p-1$}
\end{equation}
and their number is $3 p +1,$ which is in accordance with equation
 ~\eqref{eq:nbschw}. It follows from Eq. ~\eqref{eq:varphiseq} that to
find explicit expressions for the absolute invariant $\chi_k (S_j)$,
we only need to compute the corresponding relative invariant
$\varphi_k (S_j)$ and to determine its index. For $S_1,$  the
corresponding higher order relative invariants up to the order three
are
\begin{align*}
\varphi (S_1) &=  4 (a_4 - a_3') a_3' + 3 a_3 (-a_4' + a_3'') \\
\varphi_2 (S_1)& = 32 (a_4 - a_3') a_3'^2 -
  3 a_3 (9 a_4' a_3'+ (4 a_4 - 13 a_3') a_3'') \\
       &\quad  +
  9 a_3^2 (a_4'' - a_3^{(3)})
\end{align*}
and their indices are clearly $8$ and $12,$ respectively. The size
of these two invariants are much more smaller compared to those
corresponding to $S_2$ and $S_3,$ simply because $S_1$ has a much
smaller size compared to $S_2$ and $S_3.$ For instance the
expression for $\varphi_k (S_3)$ is about five times larger in size
than that for $\varphi_k (S_1),$ for $k=1,2.$ It turns out that the
expression for the invariants computed explicitly by means of the
infinitesimal generator $X^0$ are considerably much smaller in size
than those computed from the indefinite sequence
~\eqref{eq:varphiseq}. Indeed, a direct computation shows that a
fundamental set of differential invariants of order two is given by
\begin{align*}
I_4 &=  \frac{(7 a_4^2 - 14 a_4 a_3'+ 6 a_3 a_3'')^3}{216\, a_3^8}\\
I_5 &=  \frac{4 a_4^3+ 24 a_4^2 a_3'+ 9 a_3^2 a_4''- 9 a_4 a_3 (3 a_4'+ a_3'')}{9\, a_3^4}\\
I_6 &=\frac{(-18 a_4^4-18 a_4^3 a_3'+ 18 a_3^3 a_5''- 6 a_4 a_3^2
(11 a_5' + 2 a_4'')+ a_4^2 a_3 (55 a_5+ 40 a_4' + 6 a_3''))^3}{5832
 \, a_3^{16}}
\end{align*}
and that for differential invariants of order three is given by
\begin{flushleft}
\begin{align*}
I_7 &=\frac{-14 a_4^3+ 42 a_4^2 a_3'- 36 a_4 a_3 a_3'' + 9 a_3^2 a_3^{(3)}}{9\, a_3^4} \\
I_8 &= \frac{(-2 a_4^4-12 a_4^3 a_3'+ 3 a_4^2 a_3 (5 a_4'+ 3 a_3'')+
2 a_3^3 a_4^{(3)}- 2 a_4 a_3^2(5 a_4'' + a_3^{(3)}))^3}{8\, a_3^{16}} \\
I_9 &= (S_{9,1}+ S_{9,2})^3/(5832\, a_3^{20})
\intertext{ where }
S_{9,1} &= 35 a_4^5 + 45 a_4^4 a_3'- 10 a_4^3 a_3 (11 a_5+ 11 a_4'+
3
a_3'') + 18 a_3^4 a_5^{(3)}\\
S_{9,2}&= -12 a_4 a_3^3 (9 a_5''+ a_4^{(3)}+ 6 a_4^2 a_3^2 (33 a_5'
+ 11 a_4''+ a_3^{(3)})^3).
\end{align*}
\end{flushleft}
As already mentioned  these absolute invariants of ~\eqref{eq:sh5}
are given in ~\cite{ndogschw}, but only for the second prolongation
of the operator $X^0,$ and not in a form in which the index can be
readily read off. They indeed appear to be smaller in size than
those derived from the recurrence relations ~\eqref{eq:varphiseq},
which indicates that a further simplification of the expression of
the function $\varphi$ might be possible. However, such a
simplification might lead to more complicated recurrence equations
for the indefinite sequence of invariants. It's however naturally
much easier to generate higher order invariants using the indefinite
sequence ~\eqref{eq:varphiseq}, rather than finding first the right
prolongation of $X^0$ and then solving the corresponding system of
equations to find the invariants. \par
Similarly, using Eqs. ~\eqref{eq:ifsh5sm}, ~\eqref{eq:ivp1sh5}, and
 ~\eqref{eq:varphiseq}, we can also determine, by invoking Corollary
 ~\ref{co:4}, a fundamental set of relative invariants of all orders.
For instance,  fundamental relative invariants of order up to two of
 ~\eqref{eq:sh5} are given by
\begin{subequations} \label{eq:smp2sh5}
\begin{align}
\varphi_k (S_j, S_0),& \quad j=1, 2, 3, \quad k=0, 1\\
\varphi_k(R_0, S_0),&\quad \varphi_k(S_0, R_0), \quad k=0, 1, 2.
\end{align}
\end{subequations}
Incidently, Halphen ~\cite{halph66} constructed by a different
method an indefinite sequence of relative invariants similar to that
given in ~\eqref{eq:smp2sh5}, but for linear ODEs of order $4$ in
the canonical form ~\eqref{eq:gnlin}. In particular the approach of
 ~\cite{halph66} is not based on infinitesimal nor Lie groups
 methods and it aimed only
at deriving and indefinite sequence of relative invariants from
known ones. Consequently, the sequence of relative invariants
obtained in the said paper for equations of order $4$ is not
associated with the determination of a fundamental set of absolute
invariants.



%

%
\end{document}